\documentclass[a4paper,10pt]{article}
 
\usepackage{latexsym,amssymb,amsmath,amsthm,mathrsfs,euscript}  

\usepackage[width=14.5cm]{geometry}
\newtheorem{thm}{Theorem}
\newtheorem{lem}{Lemma}

\newtheorem{defn}{Definition} 
\newtheorem{rem}{Remark}

\begin{document}

\title{Fueter's theorem for monogenic functions in biaxial symmetric domains}

\author{Dixan Pe\~na Pe\~na$^{\text{a}}$\\
\small{e-mail: dixanpena@gmail.com}
\and Irene Sabadini$^{\text{a}}$\\
\small{e-mail: irene.sabadini@polimi.it}
\and Franciscus Sommen$^{\text{b}}$\\
\small{e-mail: franciscus.sommen@ugent.be}}

\date{\small{$^\text{a}$Dipartimento di Matematica, Politecnico di Milano\\Via E. Bonardi 9, 20133 Milano, Italy\\\vspace{0.2cm}
$^{\text{b}}$Clifford Research Group, Department of Mathematical Analysis\\Faculty of Engineering and Architecture, Ghent University\\Galglaan 2, 9000 Gent, Belgium}}

\maketitle

\begin{abstract}
\noindent In this paper we generalize the result on Fueter's theorem from \cite{EPVS} by Eelbode et al. to the case of monogenic functions in biaxially symmetric domains. To obtain this result, Eelbode et al. used representation theory methods but their result also follows from a direct calculus we established in our paper \cite{DS2}. In this paper we first generalize \cite{DS2} to the biaxial case and derive the main result from that.\vspace{0.2cm}\vspace{0.1cm}
\end{abstract}

\section{Introduction}

Let $\mathbb{R}_{m}$ be the real Clifford algebra generated by the standard basis $\{e_1,\ldots,e_m\}$ of the Euclidean space $\mathbb R^m$ (see \cite{Cl,Lou}). The multiplication in this  associative algebra is determined by the relations: $e_j^2=-1$, $e_je_k+e_ke_j=0$, $1\le j\neq k \le m$. Any Clifford number $a\in\mathbb{R}_{m}$ may thus be written as 
\[a=\sum_Aa_Ae_A,\quad a_A\in\mathbb R,\] 
where the basis elements $e_A=e_{j_1}\dots e_{j_k}$ are defined for every subset $A=\{j_1,\dots,j_k\}$ of $\{1,\dots,m\}$ with $j_1<\dots<j_k$ (for $A=\emptyset$ one puts $e_{\emptyset}=1$). 

Observe that $\mathbb R^{m+1}$ may be naturally embedded in $\mathbb R_{m}$ by associating to any element $(X_0,X_1,\ldots,X_m)\in\mathbb R^{m+1}$ the paravector $X_0+\underline X=X_0+\sum_{j=1}^mX_je_j$. Furthermore, by the above multiplication rules it follows that $\underline X^2=-\vert\underline X\vert^2=-\sum_{j=1}^mX_j^2$.

The even and odd subspaces $\mathbb R_{m}^+$, $\mathbb R_{m}^-$ are defined as  
\[\mathbb R_{m}^+=\Biggl\{a\in\mathbb R_{m}:\;a=\sum_{\vert A\vert\;\textrm{even}}a_Ae_A\Biggr\},\;\;\mathbb R_{m}^-=\Biggl\{a\in\mathbb R_{m}:\;a=\sum_{\vert A\vert\;\textrm{odd}}a_Ae_A\Biggr\},\]
where $\vert A\vert=j_1+\dots+j_k$. The subspace $\mathbb R_{m}^+$ is also a subalgebra and we have that
\[\mathbb R_{m}=\mathbb R_{m}^+\oplus\mathbb R_{m}^-.\]
Consider the Dirac operator $\partial_{\underline X}$ in $\mathbb R^m$ given by 
\[\partial_{\underline X}=\sum_{j=1}^me_j\partial_{X_j},\]
which provides a factorization of the Laplacian, i.e. $\partial_{\underline X}^2=-\Delta_{\underline X}=-\sum_{j=1}^m\partial_{X_j}^2$. Functions in the kernel of $\partial_{\underline X}$ are known as monogenic functions (see \cite{BDS,CoSaSo,DSS,GM,GuSp}).

\begin{defn}
A function $F:\Omega\rightarrow\mathbb{R}_{m}$ defined and continuously differentiable in an open set $\Omega\subset\mathbb R^{m}$ is said to be $($left$)$ monogenic in $\Omega$ if $\partial_{\underline X}F(\underline X)=0$, $\underline X\in\Omega$. In a similar way one defines monogenicity with respect to the generalized Cauchy-Riemann operator $\partial_{X_0}+\partial_{\underline X}$ in $\mathbb R^{m+1}$.
\end{defn}

It is clear that monogenic functions are harmonic. Furthermore, for the particular case $m=1$ the equation $(\partial_{X_0}+\partial_{\underline X})F(X_0,\underline X)=0$ is nothing but the classical Cauchy-Riemann system for holomorphic functions. This is not the only connection existing between holomorphic and monogenic functions as the following result shows (see \cite{S3}).

\begin{thm}[Fueter's theorem]\label{Ftthm}
Let $w(z)=u(x,y)+iv(x,y)$ be a holomorphic function in the open subset $\Xi$ of the upper half-plane and assume that $P_K(\underline X)$ is a homogeneous monogenic polynomial of degree $K$ in $\mathbb R^m$. If $m$ is odd, then the function
\begin{equation}\label{Ftmapp}
\big(\partial_{X_0}^2+\Delta_{\underline X}\big)^{K+\frac{m-1}{2}}\left[\left(u(X_0,\vert\underline X\vert)+\frac{\underline X}{\vert\underline X\vert}\,v(X_0,\vert\underline X\vert)\right)P_K(\underline X)\right]
\end{equation}
is monogenic in $\Omega=\big\{(X_0,\underline X)\in\mathbb R^{m+1}:\;(X_0,\vert\underline X\vert)\in\Xi\big\}$.
\end{thm}

The idea of using holomorphic functions to construct monogenic functions was first presented by Fueter \cite{Fue} in the setting of quaternionic analysis ($m=3$, $K=0$) and for that reason Theorem \ref{Ftthm} bears his name. In 1957 Sce \cite{Sce} extended Fueter's idea to Clifford analysis by proving the validity of the above result for the case $K=0$, $m$ odd. Forty years later Qian \cite{Q} showed that a similar result holds when $m$ is even. In the last years several articles have been published on this topic (see e.g. \cite{CSDPF,CoSaF,CoSaF2,CoSaF3,HDPS,EPVS,KQS,DS1,QS}). For more information we refer the reader to the survey article \cite{Q2}.

Consider the biaxial decomposition $\mathbb R^{m}=\mathbb R^{p}\oplus\mathbb R^{q}$, $p+q=m$. In this way, for any $\underline X\in\mathbb R^{m}$ we may write  
\[\underline X=\underline x+\underline y,\] 
where $\underline x=\sum_{j=1}^{p}x_je_j$ and $\underline y=\sum_{j=1}^{q}x_{p+j}e_{p+j}$. We shall denote by $\mathbb R_{p}$ and $\mathbb R_{q}$ the real Clifford algebras constructed over $\mathbb R^{p}$ and $\mathbb R^{q}$ respectively, i.e.
\[\mathbb R_{p}=\text{Alg}_{\mathbb R}\big\{e_1,\ldots,e_p\},\quad\mathbb R_{q}=\text{Alg}_{\mathbb R}\big\{e_{p+1},\ldots,e_m\}.\] 

In this paper we further investigate the following generalization of Fueter's theorem to the biaxial case (see \cite{DQiS,DxS,QS}). We note that in this setting there is a slight change regarding the initial function $w$. Namely, $w$ will be assumed to be antiholomorphic, i.e. a solution of $\partial_zw=0$, where $\partial_z=\frac{1}{2}(\partial_{x}-i\partial_{y})$.  

\begin{thm}\label{Ftthmbiaxial}
Let $w(\overline z)=u(x,y)+iv(x,y)$ be an antiholomorphic function in an open subset of $\{(x,y)\in\mathbb R^2:\;x,y>0\}$. Suppose that $P_k(\underline x):\mathbb R^{p}\rightarrow\mathbb{R}_{p}$ and $P_\ell(\underline y):\mathbb R^{q}\rightarrow\mathbb{R}_{q}$ are homogeneous monogenic polynomials. If $p$ and $q$ are odd, then the functions 
\begin{equation*}
\mathsf{Ft}_{p,q}^{+}\left[w(\overline z),P_k(\underline x),P_\ell(\underline y)\right](\underline X)=\Delta_{\underline X}^{k+\ell+\frac{m-2}{2}}\left[\left(u(\vert\underline x\vert,\vert\underline y\vert)+\frac{\underline x\,\underline y}{\vert\underline x\vert\vert\underline y\vert}\,v(\vert\underline x\vert,\vert\underline y\vert)\right)P_k(\underline x)P_\ell(\underline y)\right]
\end{equation*}
\begin{equation*}
\mathsf{Ft}_{p,q}^{-}\left[w(\overline z),P_k(\underline x),P_\ell(\underline y)\right](\underline X)=\Delta_{\underline X}^{k+\ell+\frac{m-2}{2}}\left[\left(\frac{\underline x}{\vert\underline x\vert}\,u(\vert\underline x\vert,\vert\underline y\vert)+\frac{\underline y}{\vert\underline y\vert}\,v(\vert\underline x\vert,\vert\underline y\vert)\right)P_k(\underline x)P_\ell(\underline y)\right]
\end{equation*}
are monogenic.
\end{thm}

It is remarkable that Theorem \ref{Ftthm} is still true if $P_K(\underline X)$ is replaced by a homogeneous monogenic polynomial $P_K(X_0,\underline X)$ in $\mathbb R^{m+1}$ (see \cite{DS2}), or if the monogenicity condition on $P_K(\underline X)$ is dropped. The latter result was proved in \cite{EPVS} with the help of representation theory, but it can also be derived using the results obtained in \cite{DS2}.       

Motivated by \cite{EPVS} and using similar methods as in \cite{DS2}, we prove in this paper that Theorem \ref{Ftthmbiaxial} also holds if $P_k(\underline x)$ and $P_\ell(\underline y)$ are assumed to be only homogeneous polynomials.

\section{A higher order version of Theorem \ref{Ftthmbiaxial}}

The goal in this section is to generalize Theorem \ref{Ftthmbiaxial} to a larger class of initial functions. More precisely, we shall assume that $w(z,\overline z)=u(x,y)+iv(x,y)$ is a solution of the equation 
\begin{equation}\label{initfucnt}
\partial_z\Delta_{x,y}^{\mu}w(z,\overline z)=0,\quad \Delta_{x,y}=\partial_{x}^2+\partial_{y}^2,\quad\mu\in\mathbb N_0.
\end{equation}
In particular, poly-antiholomorphic functions of order $\mu+1$ (i.e. solutions of $\partial_z^{\mu+1}w(z,\overline z)=0$) clearly satisfy equation (\ref{initfucnt}).

It is possible to compute in explicit form the monogenic function produced by Theorem \ref{Ftthm} using the differential operators 
\begin{equation}\label{specOpe}
\left(x^{-1}\frac{d}{dx}\right)^n,\quad\quad\left(\frac{d}{dx}\,x^{-1}\right)^n,\quad n\ge 0.
\end{equation}
Namely, function (\ref{Ftmapp}) equals 
\begin{equation*}
(2K+m-1)!!\left(\left(R^{-1}\partial_R\right)^{K+\frac{m-1}{2}}u(X_0,R)+\frac{\underline X}{R}\,\left(\partial_R\,R^{-1}\right)^{K+\frac{m-1}{2}}v(X_0,R)\right)P_K(\underline X),
\end{equation*} 
where $R=\vert\underline X\vert$ (see \cite{DQiS,DS1}). 

The differential operators in (\ref{specOpe}) possess interesting properties (see \cite{HDPS,DQiS,DS1}) and in this paper we shall use the following. 

\begin{lem}\label{mopedp}
If $f:\mathbb R\rightarrow\mathbb R$ is a infinitely differentiable function, then
\begin{itemize}
\item[{\rm(i)}] $\displaystyle{\frac{d^2}{dx^2}\left(x^{-1}\frac{d}{dx}\right)^n f(x)=\left(x^{-1}\frac{d}{dx}\right)^n\frac{d^2}{dx^2}f(x)-2n\left(x^{-1}\frac{d}{dx}\right)^{n+1}f(x)}$, 
\item[{\rm(ii)}] $\displaystyle{\frac{d^2}{dx^2}\left(\frac{d}{dx}\,x^{-1}\right)^n f(x)=\left(\frac{d}{dx}\,x^{-1}\right)^n\frac{d^2}{dx^2}f(x)-2n\left(\frac{d}{dx}\,x^{-1}\right)^{n+1}f(x)}$,
\item[{\rm(iii)}] $\displaystyle{\left(\frac{d}{dx}\,x^{-1}\right)^n\frac{d}{dx}f(x)=\frac{d}{dx}\left(x^{-1}\frac{d}{dx}\right)^n f(x)}$,  
\item[{\rm(iv)}] $\displaystyle{\left(x^{-1}\frac{d}{dx}\right)^n\frac{d}{dx}f(x)-\frac{d}{dx}\left(\frac{d}{dx}\,x^{-1}\right)^n f(x)=2nx^{-1}\left(\frac{d}{dx}\,x^{-1}\right)^n f(x)}$.
\end{itemize}
\end{lem}
\noindent
Due to the decomposition $\mathbb R^{m}=\mathbb R^{p}\oplus\mathbb R^{q}$ it is convenient to split $\partial_{\underline X}$ and $\Delta_{\underline X}$ as 
\begin{align*}
\partial_{\underline X}&=\partial_{\underline x}+\partial_{\underline y}=\sum_{j=1}^pe_j\partial_{x_j}+\sum_{j=1}^qe_{p+j}\partial_{x_{p+j}},\\
\Delta_{\underline X}&=\Delta_{\underline x}+\Delta_{\underline y}=\sum_{j=1}^p\partial_{x_j}^2+\sum_{j=1}^q\partial_{x_{p+j}}^2.
\end{align*}
Furthermore, for any $\underline x\in\mathbb R^p$ and $\underline y\in\mathbb R^q$ we put 
\begin{align*}
\underline\omega&=\underline x/r,\quad r=\vert\underline x\vert,\\
\underline\nu&=\underline y/\rho,\quad \rho=\vert\underline y\vert.
\end{align*}
In this section, like in Theorem \ref{Ftthmbiaxial}, we assume that $P_k(\underline x):\mathbb R^{p}\rightarrow\mathbb{R}_{p}$ and $P_\ell(\underline y):\mathbb R^{q}\rightarrow\mathbb{R}_{q}$ are homogeneous monogenic polynomials. It is convenient to make a few observations about these polynomials. 

\begin{rem}\label{rempolins}
First, note that $P_k(\underline x)$ can be uniquely written in the form $P_k(\underline x)=P_k^+(\underline x)+P_k^-(\underline x)$, where $P_k^+(\underline x)$, $P_k^-(\underline x)$ take values in $\mathbb{R}_{p}^+$, $\mathbb{R}_{p}^-$ respectively. Since $\partial_{\underline x}P_k^+(\underline x)\in\mathbb{R}_{p}^-$, $\partial_{\underline x}P_k^-(\underline x)\in\mathbb{R}_{p}^+$ for $\underline x\in\mathbb R^{p}$, one can conclude that $P_k(\underline x)$ is monogenic if and only if both components $P_k^+(\underline x)$ and $P_k^-(\underline x)$ are monogenic. Of course, a similar remark holds for $P_\ell(\underline y)$. 
\end{rem}
\noindent
Let $\Delta_2=\partial_{r}^2+\partial_{\rho}^2$ be the two-dimensional Laplacian in the variables $(r,\rho)$ and recall the definition of a multinomial coefficient
\[\binom{n}{j_1,j_2,\dots,j_s}=\frac{n!}{j_1!\,j_2!\cdots j_s!}.\]
Consider the function $D:\mathbb N_0\times\mathbb N_0\rightarrow\mathbb Z$ satisfying
\begin{align*}
D(0,0)&=1,\quad D(j_1,j_2)=D(j_1,0)D(0,j_2),\;\; j_1,j_2\ge1\\
D(j,0)&=\prod_{s=1}^{j}\big(2k+p-(2s-1)\big),\quad D(0,j)=\prod_{s=1}^{j}\big(2\ell+q-(2s-1)\big),\quad j\ge1.
\end{align*}

\begin{lem}\label{mainident}
Suppose that $h:\mathbb R^2\rightarrow\mathbb R$ is an infinitely differentiable function in an open subset of $\{(x,y)\in\mathbb R^2:\;x,y>0\}$. Then for $n\in\mathbb N$ and $s_1,s_2\in\{0,1\}$ it holds that
\begin{multline}\label{indefund}
\Delta_{\underline X}^{n}\Big(h(r,\rho)\,\underline\omega^{s_1}\underline\nu^{s_2}P_k(\underline x)P_\ell(\underline y)\Big)=\\
\left(\sum_{\substack{j_1+j_2\le n\\j_1,j_2\ge 0}}\binom{n}{j_1,j_2,n-j_1-j_2}D(j_1,j_2)W_{j_1,j_2}^{s_1,s_2}(r,\rho)\right)\underline\omega^{s_1}\underline\nu^{s_2}P_k(\underline x)P_\ell(\underline y),
\end{multline}
where
\[W_{j_1,j_2}^{0,0}(r,\rho)=\left(r^{-1}\partial_r\right)^{j_1}\left(\rho^{-1}\partial_{\rho}\right)^{j_2}\Delta_2^{n-j_1-j_2}h(r,\rho),\]
\[W_{j_1,j_2}^{1,0}(r,\rho)=\left(\partial_r\,r^{-1}\right)^{j_1}\left(\rho^{-1}\partial_{\rho}\right)^{j_2}\Delta_2^{n-j_1-j_2}h(r,\rho),\]
\[W_{j_1,j_2}^{0,1}(r,\rho)=\left(r^{-1}\partial_r\right)^{j_1}\left(\partial_{\rho}\,\rho^{-1}\right)^{j_2}\Delta_2^{n-j_1-j_2}h(r,\rho),\]
\[W_{j_1,j_2}^{1,1}(r,\rho)=\left(\partial_r\,r^{-1}\right)^{j_1}\left(\partial_{\rho}\,\rho^{-1}\right)^{j_2}\Delta_2^{n-j_1-j_2}h(r,\rho).\]
\end{lem}
\begin{proof} 
We shall prove the case $s_1=s_2=0$ using induction. The other cases can be proved similarly. First, note that
\begin{align*}
\partial_{\underline x}h&=\sum_{j=1}^pe_j\partial_{x_j}h=\sum_{j=1}^pe_j\big(\partial_rh\big)\big(\partial_{x_j}r\big)=\underline\omega\,\partial_rh
\end{align*}
and hence
\begin{align*}
\Delta_{\underline x}h&=-\partial_{\underline x}^2h=-\partial_{\underline x}\big(\underline\omega\,\partial_rh\big)=-\underline\omega^2\partial_r^2h-\big(\partial_{\underline x}\,\underline\omega\big)\big(\partial_rh\big)\\
&=\partial_r^2h+\frac{p-1}{r}\,\partial_rh.
\end{align*}
From this equality and using Euler's theorem for homogeneous functions we obtain 
\begin{align*}
\Delta_{\underline x}\big(hP_k\big)&=\big(\Delta_{\underline x}h\big)P_k+2\sum_{j=1}^p\big(\partial_{x_j}h\big)\big(\partial_{x_j}P_k\big)+h\big(\Delta_{\underline x}P_k\big)\\
&=\left(\partial_r^2h+\frac{p-1}{r}\,\partial_rh\right)P_k+2\frac{\partial_rh}{r}\sum_{j=1}^px_j\partial_{x_j}P_k=\left(\partial_r^2h+\frac{2k+p-1}{r}\,\partial_rh\right)P_k.
\end{align*}
In a similar way one also get 
\[\Delta_{\underline y}\big(hP_\ell\big)=\left(\partial_{\rho}^2h+\frac{2\ell+q-1}{\rho}\,\partial_{\rho}h\right)P_\ell.\]
These equalities then yield 
\begin{equation}\label{primeeq}
\Delta_{\underline X}\big(hP_kP_\ell\big)=\left(\Delta_2h+\frac{2k+p-1}{r}\,\partial_rh+\frac{2\ell+q-1}{\rho}\,\partial_{\rho}h\right)P_kP_\ell.
\end{equation}
It is clear that the assertion is true in the case $n=1$. Assume now that the identity holds for some natural number $n$. We thus get
\begin{multline*}
\Delta_{\underline X}^{n+1}\big(hP_kP_\ell\big)=\sum_{\substack{j_1+j_2\le n\\j_1,j_2\ge 0}}\binom{n}{j_1,j_2,n-j_1-j_2}D(j_1,j_2)\\
\times\Delta_{\underline X}\left(\left(r^{-1}\partial_r\right)^{j_1}\left(\rho^{-1}\partial_{\rho}\right)^{j_2}\Delta_2^{n-j_1-j_2}hP_kP_\ell\right).
\end{multline*}
By statement {\rm(i)} of Lemma \ref{mopedp} we obtain
\begin{multline*}\label{segueq}
\Delta_2\left(r^{-1}\partial_r\right)^{j_1}\left(\rho^{-1}\partial_{\rho}\right)^{j_2}\Delta_2^{n-j_1-j_2}h=\left(r^{-1}\partial_r\right)^{j_1}\left(\rho^{-1}\partial_{\rho}\right)^{j_2}\Delta_2^{n+1-j_1-j_2}h\\
-2j_1\left(r^{-1}\partial_r\right)^{j_1+1}\left(\rho^{-1}\partial_{\rho}\right)^{j_2}\Delta_2^{n-j_1-j_2}h-2j_2\left(r^{-1}\partial_r\right)^{j_1}\left(\rho^{-1}\partial_{\rho}\right)^{j_2+1}\Delta_2^{n-j_1-j_2}h.
\end{multline*}
This equality and (\ref{primeeq}) imply that 
\begin{multline*}
D(j_1,j_2)\Delta_{\underline X}\left(\left(r^{-1}\partial_r\right)^{j_1}\left(\rho^{-1}\partial_{\rho}\right)^{j_2}\Delta_2^{n-j_1-j_2}hP_kP_\ell\right)=\\
\bigg(D(j_1,j_2)\left(r^{-1}\partial_r\right)^{j_1}\left(\rho^{-1}\partial_{\rho}\right)^{j_2}\Delta_2^{n+1-j_1-j_2}h+D(j_1+1,j_2)\left(r^{-1}\partial_r\right)^{j_1+1}\left(\rho^{-1}\partial_{\rho}\right)^{j_2}\Delta_2^{n-j_1-j_2}h\\
+D(j_1,j_2+1)\left(r^{-1}\partial_r\right)^{j_1}\left(\rho^{-1}\partial_{\rho}\right)^{j_2+1}\Delta_2^{n-j_1-j_2}h\bigg)P_kP_\ell.
\end{multline*}
Therefore
\begin{equation}\label{iguabru}
\Delta_{\underline X}^{n+1}\big(hP_kP_\ell\big)=\big(T_1+T_2+T_3\big)P_kP_\ell,
\end{equation}
where
\[T_1=\sum_{\substack{j_1+j_2\le n\\j_1,\,j_2\ge 0}}\binom{n}{j_1,j_2,n-j_1-j_2}D(j_1,j_2)\left(r^{-1}\partial_r\right)^{j_1}\left(\rho^{-1}\partial_{\rho}\right)^{j_2}\Delta_2^{n+1-j_1-j_2}h,\]
\[T_2=\sum_{\substack{j_1+j_2\le n+1\\j_1\ge 1,\,j_2\ge 0}}\binom{n}{j_1-1,j_2,n+1-j_1-j_2}D(j_1,j_2)\left(r^{-1}\partial_r\right)^{j_1}\left(\rho^{-1}\partial_{\rho}\right)^{j_2}\Delta_2^{n+1-j_1-j_2}h,\]
\[T_3=\sum_{\substack{j_1+j_2\le n+1\\j_1\ge 0,\,j_2\ge 1}}\binom{n}{j_1,j_2-1,n+1-j_1-j_2}D(j_1,j_2)\left(r^{-1}\partial_r\right)^{j_1}\left(\rho^{-1}\partial_{\rho}\right)^{j_2}\Delta_2^{n+1-j_1-j_2}h.\]
Observe that set $\{(j_1,j_2):\,j_1+j_2\le n+1,\,j_1,j_2\ge 0\}$ can be expressed as the union of the disjoint sets $\{(0,0)\}$, $\{(n+1,0)\}$, $\{(0,n+1)\}$, $\{(j,0):\,1\le j\le n\}$, $\{(0,j):\,1\le j\le n\}$, $\{(j,n+1-j):\,1\le j\le n\}$ and $\{(j_1,j_2):\,j_1+j_2\le n,\,j_1,j_2\ge 1\}$. Taking this into account it is easy to verify that (\ref{iguabru}) equals 
\[\left(\sum_{\substack{j_1+j_2\le n+1\\j_1,j_2\ge 0}}\binom{n+1}{j_1,j_2,n+1-j_1-j_2}D(j_1,j_2)\left(r^{-1}\partial_r\right)^{j_1}\left(\rho^{-1}\partial_{\rho}\right)^{j_2}\Delta_2^{n+1-j_1-j_2}h\right)P_kP_\ell.\]
Thus proving the assertion for $n+1$. 
\end{proof}

\begin{rem}\label{com2}
Theorem \ref{Ftthmbiaxial} yields biaxial monogenic functions, i.e. monogenic functions of the form 
\[\bigl(A(r,\rho)+\underline\omega\,\underline\nu\,B(r,\rho)\bigr)P_k(\underline x)P_{\ell}(\underline y)\] 
or
\[\bigl(\underline\omega\,C(r,\rho)+\underline\nu\,D(r,\rho)\bigr)P_k(\underline x)P_{\ell}(\underline y),\] 
where $A$, $B$, $C$, $D$ are $\mathbb R$-valued continuously differentiable functions in $\mathbb R^2$ $($see $\cite{JaS,LB,S2})$. A direct computation shows that the pairs $(A,B)$ and $(C,D)$ satisfy the following Vekua-type systems
\begin{alignat*}{2}
\left\{\begin{aligned}
\partial_{r}A+\partial_{\rho}B&=-\frac{2\ell+q-1}{\rho}B\\
\partial_{\rho}A-\partial_{r}B&=\frac{2k+p-1}{r}B,
\end{aligned}\right.
\qquad
\left\{\begin{aligned}
\partial_{r}C+\partial_{\rho}D&=-\frac{2k+p-1}{r}C-\frac{2\ell+q-1}{\rho}D\\
\partial_{\rho}C-\partial_{r}D&=0.
\end{aligned}\right.
\end{alignat*}
\end{rem} 

\noindent We now come to our first main result that generalizes \cite{DS2} to the biaxial case. 
\begin{thm}\label{Ftthhigh}
Suppose that $w(z,\overline z)=u(x,y)+iv(x,y)$ is a $\mathbb C$-valued function satisfying the equation $(\ref{initfucnt})$ in the open set $\,\Xi\subset\{(x,y)\in\mathbb R^2:\;x,y>0\}$. If $p$ and $q$ are odd, then the functions
\begin{equation*}
\mathsf{Ft}_{p,q}^{\mu,+}\left[w(z,\overline z),P_k(\underline x),P_\ell(\underline y)\right](\underline X)=\Delta_{\underline X}^{\mu+k+\ell+\frac{m-2}{2}}\Big[\big(u(r,\rho)+\underline\omega\,\underline\nu\,v(r,\rho)\big)P_k(\underline x)P_\ell(\underline y)\Big],
\end{equation*}
\begin{equation*}
\mathsf{Ft}_{p,q}^{\mu,-}\left[w(z,\overline z),P_k(\underline x),P_\ell(\underline y)\right](\underline X)=\Delta_{\underline X}^{\mu+k+\ell+\frac{m-2}{2}}\Big[\big(\underline\omega\,u(r,\rho)+\underline\nu\,v(r,\rho)\big)P_k(\underline x)P_\ell(\underline y)\Big]
\end{equation*}
are monogenic in $\Omega=\big\{\underline X=\big(\underline x,\underline y\big)\in\mathbb R^{m}:\;(r,\rho)\in\Xi\big\}$.
\end{thm}
\begin{proof}
We use Lemma \ref{mainident} to compute $\mathsf{Ft}_{p,q}^{\mu,+}$ in closed form. First, note that function $w$ also satisfies the equation $\Delta_2^{\mu+1}w=0$, thus 
\[\Delta_2^{\mu+k+\ell+\frac{m-2}{2}-j_1-j_2}w=0\;\;\text{for}\;\; j_1+j_2\le k+\ell+(m-4)/2.\] 
Since $p$ and $q$ are odd we also have that $D(j_1,j_2)=0$ for $j_1\ge k+(p+1)/2$ or $j_2\ge \ell+(q+1)/2$. It follows that for $n=\mu+k+\ell+(m-2)/2$ the only term in (\ref{indefund}) which does not vanish corresponds to the case $j_1=k+(p-1)/2$, $j_2=\ell+(q-1)/2$.
Therefore
\begin{multline*}
\mathsf{Ft}_{p,q}^{\mu,+}\left[w(z,\overline z),P_k(\underline x),P_\ell(\underline y)\right](\underline X)=(2k+p-1)!!(2\ell+q-1)!!\\
\times\binom{\mu+k+\ell+\frac{m-2}{2}}{k+\frac{p-1}{2},\ell+\frac{q-1}{2},\mu}\bigl(A(r,\rho)+\underline\omega\,\underline\nu\,B(r,\rho)\bigr)P_k(\underline x)P_{\ell}(\underline y),
\end{multline*}
with
\[A=\left(r^{-1}\partial_r\right)^{k+\frac{p-1}{2}}\left(\rho^{-1}\partial_{\rho}\right)^{\ell+\frac{q-1}{2}}\Delta_2^{\mu}u,\;\; B=\left(\partial_r\,r^{-1}\right)^{k+\frac{p-1}{2}}\left(\partial_{\rho}\,\rho^{-1}\right)^{\ell+\frac{q-1}{2}}\Delta_2^{\mu}v.\]
It thus remains to prove that $(A,B)$ fulfills the first system of Remark \ref{com2}. Using statement {\rm(iii)} of Lemma \ref{mopedp} and the fact that $w$ satisfies (\ref{initfucnt}) we obtain
\[\partial_{r}A=\left(\partial_r\,r^{-1}\right)^{k+\frac{p-1}{2}}\left(\rho^{-1}\partial_{\rho}\right)^{\ell+\frac{q-1}{2}}\partial_{r}\Delta_2^{\mu}u=-\left(\partial_r\,r^{-1}\right)^{k+\frac{p-1}{2}}\left(\rho^{-1}\partial_{\rho}\right)^{\ell+\frac{q-1}{2}}\partial_{\rho}\Delta_2^{\mu}v.\]
Hence we get
\begin{align*}
\partial_{r}A+\partial_{\rho}B&=-\left(\partial_r\,r^{-1}\right)^{k+\frac{p-1}{2}}\left(\left(\rho^{-1}\partial_{\rho}\right)^{\ell+\frac{q-1}{2}}\partial_{\rho}\Delta_2^{\mu}v-\partial_{\rho}\left(\partial_{\rho}\,\rho^{-1}\right)^{\ell+\frac{q-1}{2}}\Delta_2^{\mu}v\right)\\
&=-\frac{2\ell+q-1}{\rho}\left(\partial_r\,r^{-1}\right)^{k+\frac{p-1}{2}}\left(\partial_{\rho}\,\rho^{-1}\right)^{\ell+\frac{q-1}{2}}\Delta_2^{\mu}v,
\end{align*}
where we have also used statement {\rm(iv)} of Lemma \ref{mopedp}. In a similar fashion, it can be shown that
\begin{equation*}
\partial_{\rho}A-\partial_{r}B=\frac{2k+p-1}{r}\left(\partial_r\,r^{-1}\right)^{k+\frac{p-1}{2}}\left(\partial_{\rho}\,\rho^{-1}\right)^{\ell+\frac{q-1}{2}}\Delta_2^{\mu}v.
\end{equation*}
The proof of $\mathsf{Ft}_{p,q}^{\mu,-}$ goes along the same lines as that of $\mathsf{Ft}_{p,q}^{\mu,+}$. Indeed, it follows from Lemma \ref{mainident} that
\begin{multline*}
\mathsf{Ft}_{p,q}^{\mu,-}\left[w(z,\overline z),P_k(\underline x),P_\ell(\underline y)\right](\underline X)=(2k+p-1)!!(2\ell+q-1)!!\\
\times\binom{\mu+k+\ell+\frac{m-2}{2}}{k+\frac{p-1}{2},\ell+\frac{q-1}{2},\mu}\bigl(\underline\omega\,C(r,\rho)+\underline\nu\,D(r,\rho)\bigr)P_k(\underline x)P_{\ell}(\underline y),
\end{multline*}
with
\[C=\left(\partial_r\,r^{-1}\right)^{k+\frac{p-1}{2}}\left(\rho^{-1}\partial_{\rho}\right)^{\ell+\frac{q-1}{2}}\Delta_2^{\mu}u,\;\; D=\left(r^{-1}\partial_r\right)^{k+\frac{p-1}{2}}\left(\partial_{\rho}\,\rho^{-1}\right)^{\ell+\frac{q-1}{2}}\Delta_2^{\mu}v.\]
One can check, using statements {\rm(iii)} and {\rm(iv)} of Lemma \ref{mopedp} as well as equation (\ref{initfucnt}), that $(C,D)$ satisfies the second system of Remark \ref{com2}. 
\end{proof}

\section{Fueter's theorem with general homogeneous factors}

We arrive at the third and last section of the paper where we shall prove our main result. In the proof we use Theorem \ref{Ftthhigh} and the well-known Fischer decomposition (see \cite{DSS}):

\begin{thm}[Fischer decomposition]
Every homogeneous polynomial $H_K(\underline X)$ of degree $K$ in $\mathbb R^m$ admits the following decomposition
\[H_K(\underline X)=\sum_{n=0}^K\underline X^nP_{K-n}(\underline X),\]
where each $P_{K-n}(\underline X)$ is a homogeneous monogenic polynomial. 
\end{thm}

\begin{thm}
Let $w(\overline z)=u(x,y)+iv(x,y)$ be an antiholomorphic function in the open set $\,\Xi\subset\{(x,y)\in\mathbb R^2:\;x,y>0\}$. Suppose that $H_k(\underline x):\mathbb R^{p}\rightarrow\mathbb{R}_{p}$ and $H_\ell(\underline y):\mathbb R^{q}\rightarrow\mathbb{R}_{q}$ are homogeneous polynomials. If $p$ and $q$ are odd, then the functions 
\[\mathsf{Ft}_{p,q}^{+}\left[w(\overline z),H_k(\underline x),H_\ell(\underline y)\right](\underline X)\quad\text{and}\quad\mathsf{Ft}_{p,q}^{-}\left[w(\overline z),H_k(\underline x),H_\ell(\underline y)\right](\underline X)\] 
are monogenic in $\Omega=\big\{\underline X=\big(\underline x,\underline y\big)\in\mathbb R^{m}:\;(r,\rho)\in\Xi\big\}$.
\end{thm}
\begin{proof}
We only prove the statement for function $\mathsf{Ft}_{p,q}^{+}$ since the proof for $\mathsf{Ft}_{p,q}^{-}$ is similar. Note that the Fischer decomposition ensures the existence of homogeneous monogenic polynomials $P_{k-n_1}(\underline x)$ and $P_{\ell-n_2}(\underline y)$ such that
\[H_k(\underline x)H_\ell(\underline y)=\sum_{n_1=0}^k\sum_{n_2=0}^{\ell}\underline x^{n_1}P_{k-n_1}(\underline x)\,\underline y^{n_2}P_{\ell-n_2}(\underline y).\]
This gives
\begin{multline*}
\mathsf{Ft}_{p,q}^{+}\left[w(\overline z),H_k(\underline x),H_\ell(\underline y)\right](\underline X)\\
=\sum_{n_1=0}^k\sum_{n_2=0}^{\ell}\Delta_{\underline X}^{k+\ell+\frac{m-2}{2}}\Big[\big(u(r,\rho)+\underline\omega\,\underline\nu\,v(r,\rho)\big)\underline x^{n_1}P_{k-n_1}(\underline x)\,\underline y^{n_2}P_{\ell-n_2}(\underline y)\Big].
\end{multline*}
It will thus be sufficient to prove the monogenicity of each term in the previous sum. On account of Remark \ref{rempolins} we may assume without loss of generality that $P_{k-n_1}(\underline x)$ takes values in $\mathbb{R}_{p}^+$ and hence 
\[\underline x^{n_1}P_{k-n_1}(\underline x)\,\underline y^{n_2}P_{\ell-n_2}(\underline y)=\underline x^{n_1}\underline y^{n_2}P_{k-n_1}(\underline x)P_{\ell-n_2}(\underline y).\] 
It is easy to verify that if $n_1+n_2$ is even, then
\begin{multline*}
\Delta_{\underline X}^{k+\ell+\frac{m-2}{2}}\Big[\big(u(r,\rho)+\underline\omega\,\underline\nu\,v(r,\rho)\big)\underline x^{n_1}\underline y^{n_2}P_{k-n_1}(\underline x)P_{\ell-n_2}(\underline y)\Big]\\
=\mathsf{Ft}_{p,q}^{n_1+n_2,+}\left[w(\overline z)h^+(x,y),P_{k-n_1}(\underline x),P_{\ell-n_2}(\underline y)\right](\underline X),
\end{multline*}
where $h^+(x,y)=\left\{\begin{array}{ll}(-1)^{\frac{n_1+n_2}{2}}x^{n_1}y^{n_2}&\text{for}\;\;n_1,n_2\;\;\text{even}\\(-1)^{\frac{n_1+n_2-2}{2}}i\,x^{n_1}y^{n_2}&\text{for}\;\;n_1,n_2\;\;\text{odd}.\end{array}\right.$

\noindent Similarly, if $n_1+n_2$ is odd, then
\begin{multline*}
\Delta_{\underline X}^{k+\ell+\frac{m-2}{2}}\Big[\big(u(r,\rho)+\underline\omega\,\underline\nu\,v(r,\rho)\big)\underline x^{n_1}\underline y^{n_2}P_{k-n_1}(\underline x)P_{\ell-n_2}(\underline y)\Big]\\
=\mathsf{Ft}_{p,q}^{n_1+n_2,-}\left[w(\overline z)h^-(x,y),P_{k-n_1}(\underline x),P_{\ell-n_2}(\underline y)\right](\underline X),
\end{multline*}
where $h^-(x,y)=\left\{\begin{array}{ll}(-1)^{\frac{n_1+n_2-1}{2}}x^{n_1}y^{n_2}&\text{for}\;\;n_1\;\text{odd},\; n_2\;\text{even}\\(-1)^{\frac{n_1+n_2-1}{2}}i\,x^{n_1}y^{n_2}&\text{for}\;\;n_1\;\text{even},\; n_2\;\text{odd}.\end{array}\right.$

\noindent Clearly, $h^{\pm}$ satisfies $\partial_z^{n_1+n_2+1}h^{\pm}=0$ and for that reason 
\[\partial_z^{n_1+n_2+1}\big(w(\overline z)h^{\pm}(x,y)\big)=w(\overline z)\partial_z^{n_1+n_2+1}h^{\pm}(x,y)=0.\]
Consequently, the functions $w(\overline z)h^{\pm}(x,y)$ are solutions of (\ref{initfucnt}) for $\mu=n_1+n_2$. The result now follows from Theorem \ref{Ftthhigh}.
\end{proof}
\noindent We conclude with some examples involving the homogeneous polynomials $H_k(\underline x)=\langle\underline x,\underline t\rangle^k$, $H_\ell(\underline y)=\langle\underline y,\underline s\rangle^{\ell}$, where $\underline t\in\mathbb R^p$ and $\underline s\in\mathbb R^q$ are arbitrary fixed vectors. In order to avoid too long computations we have chosen the cases $p=q=3$, $k=1,2$ and $\ell=1$.
\begin{multline*}
\mathsf{Ft}_{3,3}^{+}\left[\overline{z}^5,\langle\underline x,\underline t\rangle,\langle\underline y,\underline s\rangle\right](\underline X)=\frac{10}{r^3}\langle\underline x,\underline t\rangle\langle\underline y,\underline s\rangle+\frac{6\,\underline x\,\underline y}{r^5}\langle\underline x,\underline t\rangle\langle\underline y,\underline s\rangle-\frac{2\,\underline t\,\underline y}{r^3}\langle\underline y,\underline s\rangle\\
+\frac{(5r^2+3\rho^2)\underline x\,\underline s}{r^5}\langle\underline x,\underline t\rangle-\frac{(5r^2+\rho^2)\underline t\,\underline s}{r^3}
\end{multline*}
\begin{equation*}
\mathsf{Ft}_{3,3}^{+}\left[\overline{z}^8,\langle\underline x,\underline t\rangle,\langle\underline y,\underline s\rangle\right](\underline X)=10\langle\underline x,\underline t\rangle\langle\underline y,\underline s\rangle-2\,\underline t\,\underline y\langle\underline y,\underline s\rangle+2\,\underline x\,\underline s\langle\underline x,\underline t\rangle+(r^2-\rho^2)\underline t\,\underline s
\end{equation*}
\begin{multline*}
\mathsf{Ft}_{3,3}^{+}\left[\overline{z}^{10},\langle\underline x,\underline t\rangle,\langle\underline y,\underline s\rangle\right](\underline X)=140(r^2-\rho^2)\langle\underline x,\underline t\rangle\langle\underline y,\underline s\rangle-56\,\underline x\,\underline y\langle\underline x,\underline t\rangle\langle\underline y,\underline s\rangle\\-4(7r^2-5\rho^2)\underline t\,\underline y\langle\underline y,\underline s\rangle+4(5r^2-7\rho^2)\underline x\,\underline s\langle\underline x,\underline t\rangle+(5r^4-14r^2\rho^2+5\rho^4)\underline t\,\underline s
\end{multline*}
\begin{multline*}
\mathsf{Ft}_{3,3}^{-}\left[i\overline{z}^6,\langle\underline x,\underline t\rangle,\langle\underline y,\underline s\rangle\right](\underline X)=\frac{2\,\underline x}{\rho^3}\langle\underline x,\underline t\rangle\langle\underline y,\underline s\rangle-\frac{(3r^2+5\rho^2)\underline y}{\rho^5}\langle\underline x,\underline t\rangle\langle\underline y,\underline s\rangle\\
+\frac{(r^2+5\rho^2)}{\rho^3}\left(\underline t\langle\underline y,\underline s\rangle+\underline s\langle\underline x,\underline t\rangle\right)
\end{multline*}
\begin{equation*}
\mathsf{Ft}_{3,3}^{-}\left[\overline{z}^9,\langle\underline x,\underline t\rangle,\langle\underline y,\underline s\rangle\right](\underline X)=2(\underline x-\underline y)\langle\underline x,\underline t\rangle\langle\underline y,\underline s\rangle+(r^2-\rho^2)\left(\underline t\langle\underline y,\underline s\rangle+\underline s\langle\underline x,\underline t\rangle\right)
\end{equation*}
\begin{multline*}
\mathsf{Ft}_{3,3}^{-}\left[\overline{z}^{11},\langle\underline x,\underline t\rangle^2,\langle\underline y,\underline s\rangle\right](\underline X)=8\left(5\underline x-7\underline y\right)\langle\underline x,\underline t\rangle^2\langle\underline y,\underline s\rangle-4(7r^2-5\rho^2)\vert\underline t\vert^2\underline y\langle\underline y,\underline s\rangle\\
+4(5r^2-7\rho^2)\left(2\underline t\langle\underline x,\underline t\rangle\langle\underline y,\underline s\rangle+\vert\underline t\vert^2\underline x\langle\underline y,\underline s\rangle+\underline s\langle\underline x,\underline t\rangle^2\right)+(5r^4-14r^2\rho^2+5\rho^4)\vert\underline t\vert^2\underline s
\end{multline*}

\subsection*{Acknowledgments}

D. Pe\~na Pe\~na acknowledges the support of a Postdoctoral Fellowship given by Istituto Nazionale di Alta Matematica (INdAM) and cofunded by Marie Curie actions.

\end{document}